\documentclass[11pt]{amsart}
\usepackage{amsmath,amssymb,amsbsy,amsfonts,amsthm,latexsym,amsopn,amstext,amsxtra,epic, euscript,amscd,indentfirst,verbatim,graphicx,xcolor}
\usepackage[margin=1.2in]{geometry}
\begin{document}

\newtheorem{theorem}{Theorem}
\newtheorem*{theorem*}{Theorem}
\newtheorem{conjecture}[theorem]{Conjecture}
\newtheorem*{conjecture*}{Conjecture}
\newtheorem{proposition}[theorem]{Proposition}
\newtheorem{question}[theorem]{Question}
\newtheorem{lemma}[theorem]{Lemma}
\newtheorem*{lemma*}{Lemma}
\newtheorem{cor}[theorem]{Corollary}
\newtheorem*{obs*}{Observation}
\newtheorem{condition}{Condition}
\newtheorem{definition}{Definition}
\newtheorem{proc}[theorem]{Procedure}
\newcommand{\comments}[1]{} 
\def\Z{\mathbb Z}
\def\Za{\mathbb Z^\ast}
\def\Fq{{\mathbb F}_q}
\def\R{\mathbb R}
\def\N{\mathbb N}
\def\C{\mathbb C}
\def\k{\kappa}
\def\grad{\nabla}

\newcommand{\todo}[1]{\textbf{\textcolor{red}{[To Do: #1]}}}

\title[Drift Laplacians with bounded drift]{Eigenvalue estimates without Bakry-Emery-Ricci bounds}
 \author[Iowa State University]{Gabriel Khan} 
 
\email{gkhan@iastate.edu}

\date{\today}

\maketitle

\begin{abstract}
We establish a lower bound for the real eigenvalues of a Laplace-Beltrami operator with an $L^\infty$-drift term. We make no assumptions that the operator is self-adjoint or that the drift has any additional regularity. In the case where the operator is self-adjoint, this establishes a lower bound on the spectrum without assuming a lower bound for the Bakry-Emery Ricci tensor. Put colloquially, this result states that no matter which way the wind blows, heat will diffuse at a definite rate depending only on the geometry of the underlying space and the maximal wind speed.
\end{abstract}

\section{Introduction}

This paper studies the spectrum of Laplace-Beltrami operators which are deformed by a bounded drift term. Our main result is to establish the following estimate on the real eigenvalues.

\begin{theorem} \label{Main theorem}
Let $(M^n, g)$ be a compact Riemannian manifold and $\Omega$ a smooth domain in $M$ (or possibly all of $M$). Suppose that $v$ is a one-form with $\| v \|_{L^\infty} < C$ and that there exists $\lambda$ real and $u \in W^{2,p} (\Omega)$ satisfying 
 \begin{equation} \label{Drifted Laplacian}
  \begin{cases} \Delta u+ v (\nabla u) = \lambda u  & x \in \Omega \\
   u(x) \equiv 0 &   x \in \partial \Omega
   \end{cases}
   \end{equation}
   
   Then there exists a constant $\delta>0$ depending only on $\| Ric \|,$  $C,~ diam(M), ~inj(M)$ and $n$ so that $\lambda > \delta$. Here, $\| Ric \|,$ is the norm of the Ricci tensor, $diam(M)$ is the diameter of $M$ and $inj(M)$ is the injectivity radius of $M$.
\end{theorem}

This estimate immediately implies a lower bound of the principal eigenvalue of drift-Laplacians on smooth bounded domains in Riemannian manifolds.

\begin{cor}
Let $(M^n, g)$ be a compact Riemannian manifold and $\Omega \subset M$ be a smooth domain with non-empty boundary. Let $v$ be some a one-form on $\Omega$ satisfying $\| v \|_\infty < C$. Consider the principle eigenvalue $\lambda_1$ of the operator $\Delta+ v( \nabla \cdot)$ on $\Omega$.
 Then there exists some constant $\delta>0$ depending only on $\| Ric \|,$  $C,~ diam(M),~ inj(M)$ and $n$ so that $\lambda > \delta$.
\end{cor}

When $v= d\varphi$ for some potential function $\varphi$, the operator $\Delta + v(\nabla \cdot)$ is self-adjoint, and Theorem \ref{Main theorem} gives lower bounds on the spectrum of in terms of bounds on the Ricci curvature and the Lipschitz constant for the potential, but does not require any assumptions on the Bakry-Emery Ricci tensor.

The basic strategy of the proof is to adapt the Li-Yau estimate \cite{LY} to the drifted Laplacian. However, the lack of a priori bound on $\nabla v$ prevents us from directly using this technique. The key insight around this roadblock is an ansatz due to Hamel, Nadirashvili, and Russ \cite{HNR} which shows that when the principle eigenvalue is minimized, the problem becomes much more regular. Intuitively, to slow diffusion as much as possible, all of the drift needs to be working in unison.

To give a informal analogy, a fireplace will do a poor job heating a cold room if you pump air towards the fireplace and the hot air escapes up a chimney. However, in this set-up, the airflow will be continuous away from the fireplace. The a posteriori regularity on the air current makes it possible to prove a gradient estimate for the temperature to show that the room still warms at a definite rate. Although this idea is straightforward conceptually, formalizing it requires some effort.

\subsection{Acknowledgments}

The author would like to thank Fangyang Zheng for his mentorship and Adrian Lam for his help with the analytic aspects of this paper. Thanks also to Kori Khan for her help editing. This work was partially supported by DARPA/ARO Grant W911NF-16-1-0383 (PI: Jun Zhang, University of Michigan). This manuscript is the part of the author’s Ph.D. dissertation. A previous version of this paper was posted on the Arxiv with the title ``On the spectrum of $L^\infty$-drifted Laplace-Beltrami operators."

\section{Background}

The study of eigenvalues on domains and manifold has a long and rich history. Classically, this is related to the problem of ``hearing the shape of a drum" \cite{Kac}, which asks whether the geometry of a space is uniquely determined by the spectrum of its Laplacian. The answer to this question is negative \cite{milnor1964eigenvalues}, but the spectrum provides rich geometric information which has applications in many mathematical fields. For an introduction, we highly recommend the lecture notes by Canzani \cite{YC} and for some applications and interesting connections, we recommend the book by Rosenberg \cite{Rosenberg}.

Standard elliptic theory shows that for a bounded open domain or a compact manifold, there is a principle eigenvalue which is positive and defines the bottom of the spectrum. However, obtaining (non-zero) lower bounds for this eigenvalue is a difficult problem, even for Riemann surfaces. For instance, it is still an open question whether there exists hyperbolic surfaces of arbitrarily large genus whose principle eigenvalue is close to $\frac{1}{4}$. One striking result of Mirzakhani shows that for ``most" hyperbolic surfaces, the principle eigenvalue is greater than $2 \times 10^{-3}$. More precisely, the probability of the principle eigenvalue of a random hyperbolic surface being smaller than this goes to zero as the genus gets large \cite{mirzakhani2013growth}. This work uses some deep analysis of the geodesics of a Weil-Petersson random surface to bound the Cheeger constant \cite{cheeger1969lower} from below.

For more general Riemannian manifolds, estimating the Cheeger constant is not feasible, so the main strategy to find lower bounds on the spectrum is to use the Li-Yau estimate \cite{LY}. The original paper of Li and Yau studied the heat equation associated to the Laplace-Beltrami operator and derived a gradient estimate for solutions to the heat equation.
It it possible to adapt this estimate to obtain gradient bounds for eigenfunctions of the Laplace-Beltrami operator, which provides a lower bound for $\lambda$. This estimate involves a lower bound on the Ricci curvature, the diameter, and the dimension of the manifold. There has been a concerted effort to sharpen these estimates to find tighter bounds (see, e.g., \cite{ZY}). Beyond eigenvalue estimates, the Li-Yau estimate has played a central role in the development of geometric analysis (most famously, it has an important role in the analysis of Ricci flow \cite{Perelman}).

Our focus is on Laplacians with a lower order drift-term, which can be thought of as a ``convection" term. Historically, drift-Laplacians have primarily been studied when the operator is self-adjoint or when the drift is divergence-free. The study of self-adjoint drift-Laplacians play a central role in the analysis of metric-measure spaces and diffusion processes (see, e.g. \cite{lott2003some,lott2009ricci,wei2009comparison}) and is important in the analysis of Ricci solitons. Furthermore, these operators are unitarily equivalent to Witten Laplacians, which were studied by Witten \cite{EW} in his work on Morse theory. The problem of understanding the spectrum of self-adjoint drift Laplacians has been studied in depth, and sharp eigenvalue bounds are known  (see, e.g., \cite{andrews2012eigenvalue,charalambous2015eigenvalue,colbois2015eigenvalues,futaki2013first}) under the assumption of a lower bound on the Bakry-Emery Ricci tensor \cite{BakryEmery}, which is defined as $Ric + \nabla^2 \varphi$ where $v = d \varphi$.


If the drift is instead assumed to be divergence-free, there is a famous result of Nash which shows that the fundamental solution of the heat kernel decays at least as quickly as the kernel of the heat equation \cite{nash1958continuity}. In other words, stirring an incompressible fluid tends to accelerate diffusion and cannot slow it down.

For general elliptic equations with rough coefficients and non-smooth domains, much of the progress on eigenvalue estimates uses sub-solutions and other techniques from partial differential equation. For instance, given an elliptic operator $L$, the work of Berestycki et. al \cite{BNV} defines the principal eigenvalue and provides positive lower bounds on the it given sub-solutions to the problem $L u \leq 0$ (and various other hypothesis). For domains in manifolds, it is generally not possible to find sub-solutions explicitly, so we are forced to estimate the eigenvalues using the coefficients and the geometry of the domain alone. However, Berestycki et. al \cite{BNV} plays an essential role in our analysis. For instance, it allows us to define the principle eigenvalue of a uniformly elliptic operator $L$:
\begin{equation}
\lambda_1 = \sup \{ \lambda ~|~ \exists u >0 \textrm{ in $\Omega$ satisfying } (L+ \lambda) u \leq 0 \}
\end{equation}
The principal eigenvalue is well-defined for a very general class of elliptic operators and in some sense provides the bottom of the spectrum for the operator. 

For domains in Euclidean space with $L^\infty$-drifts, two papers by Hamel, Nadirashvili, and Russ \cite{HNR} \cite{HNR2} provided a key breakthrough in the analysis of the principle eigenvalue. More specifically, they proved a version of the Faber-Krahn inequality for a drifted Laplacian when the drift is bounded. To show this, they start by making the key observation that when the eigenvalue is minimized, the drift takes a special form which produces much more regularity for free. This idea plays a central role in our argument, and allows the Li-Yau estimate to be used. We should note that this is not the first paper to use the Hamel-Nadirashvili-Russ heuristic to find eigenvalue estimates on Riemannian manifolds. Recently, Ferreira and Salavessa \cite{FS} used these ideas to compare the eigenvalues of $V$-Laplacians on geodesic balls to those on model spaces. Our two approaches are completely different, but the results have some overlap. In particular, Theorems 1 and 2 of their paper proves a lower bounds of the principle eigenvalue on geodesic balls.

Interesting, this work and their work both have relative advantages. Ferreira and Salavessa establish Faber-Krahn type inequalities in geodesic balls, which we are unable to do. Also, under the assumption of bounded radial sectional curvature, they are able to relax the assumption on the drift. However, our work makes no assumption that $\Omega$ is diffeomorphic to an open set in Euclidean space and only requires bounds on the Ricci curvature instead of the sectional curvature\footnote{ Theorem 2 of Ferreira-Salavessa uses the Ricci curvature but assumes the drift is radial in a geodesic ball.}. It would be of interest to determine whether it is possible to synthesize these two approaches to establish stronger results, but we will not consider that in this paper.

We only consider the real and non-zero elements of the spectrum. This set is guaranteed to be non-empty in two important cases.

\begin{enumerate} 
\item If $\Omega$ is an open subset of a compact manifold whose boundary is smooth and non-empty, then there will be a real and positive principal eigenvalue for the Dirichlet problem on $\Omega$ \cite{BNV}. Our work gives lower bounds on this eigenvalue without making any further assumptions on the boundary of $\Omega$ or assuming that we can construct a sub-solution.

\item If $v = d \varphi$ for some function $\varphi$, then the drift-Laplacian is self-adjoint and its spectrum is entirely real. The assumption that $v$ is bounded is equivalent to the assumption that $\varphi$ is uniformly Lipschitz. For now, we must assume a bound on the norm of the Ricci tensor (i.e. a two-sided Ricci curvature bound). However, this bound is only used to be able to apply the Calderon-Zygmund inequality on our manifold. If it is possible to avoid this step, then the only curvature input needed for the estimate is a lower bound on the Ricci curvature.
\end{enumerate}

\section{The proof of Theorem 1}

In this section, we provide the proof to Theorem \ref{Main theorem}. 
We start with a brief overview of the proof. In the following subsections, we will then fill in the detailed argument.



\begin{enumerate}

\item We start by making use of a Calderon-Zygmund inequality for manifolds which was proven by G\"uneysu and Pigola \cite{GP} to derive a $W^{2,p}$ estimate on $u$ satisfying \eqref{Drifted Laplacian}. For this, we need a bound on the norm of the Ricci tensor and some lower-order geometry. This is the only step of the proof where we make use of the Ricci curvature upper bound.


We then use the Ricci curvature, injectivity radius and volume to find an atlas on $M$ with bounded $C^\alpha$ harmonic radius \cite{AC}. Using a partition of unity subordinate to this atlas, we obtain a $C^{1,\alpha}$ estimate on $u$ using Morrey's inequality on each chart.
In spirit, this is similar to Theorem 7.1 of \cite{Coulhon}, but for compact manifolds.
This part of the argument is general, and not specialized to the problem at hand. However, the estimates must be done carefully so that they depend on the geometry in a minimal way.

\item We consider the domain on which the function $u$ is positive. We expand this domain if need be so that the boundary is smooth and consider a sequence of drifts that minimize the principle eigenvalue $\lambda$ on that domain. We pick some subsequence for which the associated drifts and the corresponding eigenfunctions converges in some weak sense. When the drift minimizes $\lambda$, we find that the minimizing function satisfies the semi-linear equation
 \begin{equation}
     \Delta u + C | \nabla u |  + \lambda u= 0
 \end{equation}  
  with Dirichlet conditions. This phenomena was first observed by Hamel, Nadirashvili and Russ \cite{HNR} and essentially provides $C^1$ control over the drift away from the zero locus of $u$ and $\nabla u$.

\item We then use standard elliptic theory (Theorem 6.2 of Gilbarg-Trudinger \cite{GT} and our $W^{2, p}$ estimate on $u$) to bootstrap the regularity of $u$ to $C^{3, \alpha}$ in a small neighborhoods away from the zero locus of $\nabla u$ and $u$. We refer to these estimates as the Schauder bounds. Since these bounds are heavily dependent on the neighborhood we are working in, we cannot incorporate them into our estimate of $\lambda$ (if we try to do so, the argument becomes circular). This step should be understood as a qualitative $C^3$ estimate, but not a quantitative estimate. With this $C^3$ estimate, we are able to use the Bochner technique. 

\item We consider the point $x_0 \in M$ which maximizes \begin{equation}
     F_\beta(x)= \dfrac{|\nabla  (u)|^2}{(\beta- u)^2} \psi(x)
\end{equation} where $\psi$ is a suitably chosen cut-off function. We then use a Li-Yau-type estimate to obtain an upper bound for $F_\beta(x)$. This part of the argument is mostly an adaptation of the original Li-Yau estimate for the Laplace-Beltrami operator. However, it involves a lengthy calculation.
 
  \item The Li-Yau estimate provides a bound on $|\nabla  u|$, which we integrate along a particular geodesic to obtain a Harnack estimate. Using the geometry of the manifold and the magnitude of the drift, we choose the parameter $\beta$ appropriately to obtain a lower bound on $\lambda$. 

\end{enumerate}

Before we move on to the detailed argument, let us make several remarks. First, under weaker assumptions on the drift, it is possible to show that $\lambda > 0$ using a basic maximum principle argument (see the proof of Theorem 6 in \cite{GK}), but this does not provide quantitative lower bounds, which is the goal of our argument.

Second, the bounds that we obtain will be semi-explicit. In other words, we will provide a closed form expression, but it will rely on constants which were derived in the work of G\"uneysu and Pigola \cite{GP} and Anderson and Cheeger \cite{AC}. Once those constants are known, our bound is explicit and given in closed form.

 \subsection{The $C^{1,\alpha}$ estimate}
 We start by establishing an a priori $C^{1,\alpha}$ estimate on $u$ that depends only on the Ricci curvature and the lower order geometry of $M$. To do so, we apply a Calderon-Zygmund estimate proved in the recent work of G\"uneysu and Pigola \cite{GP}.
 
 \begin{theorem*}[\cite{GP}]
  Let $1< p < \infty$ and assume that $M$ has bounded Ricci curvature and a positive injectivity radius.
 Then, for all $\varphi \in C_c^\infty(M)$,
\begin{equation} \label{Calderon-Zygmund}
  \| \nabla^2 \varphi \|_{L^p} \leq C_1 \| \varphi \|_{L^p} + C_2 \| \Delta \varphi \|_{L^p}, \end{equation}
where the constants depend only on $dim~ M, ~p,~|Ric|$ and the injectivity radius.
\end{theorem*}

G\"uneysu and Pigola's work proves this estimate in the non-compact case, but it is straightforward to adapt their result to the compact case. To do so, one uses the bound on the injectivity radius and Ricci tensor to obtain a lower bound on the $C^{1,\alpha}$ harmonic radius of precision 2 (see the appendix of \cite{GP}). From this, one can take a cover of $M$ by balls of half this radius and apply Lemma 4.8 to find a finite cover whose intersection multiplicity is bounded. In each chart, applying Theorem 3.16 obtains a $W^{2,p}$ estimate and the bounded intersection multiplicity allows one to use these local estimates to obtain a global $W^{2,p}$ estimate. After this and applying Proposition 3.12a to eliminate the gradient term, one has the desired result.

It is worth noting that with a lower bound on the Ricci tensor as well as bounds on the volume and injectivity radius, there is a lower bound on the $C^\alpha$ harmonic radius as well as bounds on the number of charts and their multiplicity \cite{AC}. To estimate the symbol of the Laplace-Beltrami operator in a coordinate chart, one needs an estimate of the following form:
\begin{equation}
     Q^{-1} \delta_{ij} \leq g_{ij} \leq Q \delta_{ij}
\end{equation}

Such an estimate is guaranteed within the $C^\alpha$ harmonic radius of precision $Q$. Therefore, it seems likely that one can derive a similar estimate with only a lower bound on the Ricci tensor. However, the main technical obstruction to this approach is that with only $C^\alpha$ control of $g$ in the coordinate charts, we do not have control of the lower order terms in the Laplace-Beltrami operator. As such, we use a two sided bound on the Ricci tensor, which gives bounds on the $C^{1,\alpha}$ harmonic radii.

  We are trying to find lower bounds on $\lambda$, so we assume that $\lambda < 1$ (if not, then $1$ is trivially a lower bound).
Taking $\varphi$ to be the function $u$ in \eqref{Drifted Laplacian}, Inequality \eqref{Calderon-Zygmund} implies the following estimate.
 \begin{eqnarray*}
  \| \nabla^2 (u) \|_{L^p} & \leq & C_1 \| u \|_{L^p} +  C_2 \| \Delta (u) \|_{L^p}  \\
  					& =  &  C_1 \| u \|_{L^p} +  C_2 \| v(\nabla u) +\lambda u \|_{L^p}  \\
					& \leq & C_1 2 \| u \|_{L^p} +  C_2 \cdot C \| (\nabla u)\|_{L^p} 
 \end{eqnarray*}

To eliminate the gradient term, we once again use Proposition 3.12a of \cite{GP}. Doing so, we find that
\[  \| \nabla^2 (u) \|_{L^p} \leq C_3 \|u \|_{L^p}. \]

Normalizing $u$ so that $\sup u = 1$, we can use the volume comparison theorem along with our Ricci and diameter estimate to get a uniform estimate on the $L^p$ norm of $u$. From this, we obtain a uniform $W^{2,p}$ estimate on $u$ that depends only on $p$, $n$, diameter, the injectivity radius, and the bounds on the Ricci curvature. This bound provides a uniform $L^p$ estimate on $\nabla |\nabla u|$. To eliminate the dependence on $p$, we set $p=2n$ (this choice is arbitrary).

We use the results of Anderson and Cheeger \cite{AC} to cover $M$ with a finite atlas of precision 2 $C^{1/2}$ harmonic coordinate charts \begin{equation}
    \phi_i : U_i \to B_{r_h} (0) \subset \mathbb{R}^n.
\end{equation} In each of these charts, we can use the precision estimates to obtain a $W^{2,p}$ bound on $u \circ \phi_i^{-1}$. From this, we can use Morrey's inequality on each ball to obtain a uniform $C^{1,\alpha}$ bound on $u \circ \phi_i^{-1}$. Using the precision estimates again, we obtain a uniform $C^{1,\alpha}$ estimate on $u$.
Therefore, for some $C_4(n,\alpha, | Ric |, diam(M), inj(M))$, we have the estimate
\[ \| u \|_{C^{1,\alpha}} < C_4. \]
 

 It is worth mentioning that we could have derived the Morrey-type estimate in the atlas of $C^{1,\alpha}$ 2-precise harmonic coordinate charts that were used to prove the Calderon-Zygmund estimate. However, we chose $C^\alpha$ 2-precise charts so that the estimate would only depend on a lower bound of the Ricci curvature. The reason to do this is that the only place the Ricci upper bound is needed is the a priori $W^{2,p}$ estimate. If we can find a way to establish this in a different way, the result will not rely on a Ricci upper bound.

 \subsection{Finding the drift that minimizes the principle eigenvalue}

When $\Omega$ is a closed compact manifold, we want to reduce our problem to a Dirichlet problem on a subdomain. To do so, consider the open manifold $M^+ = \{M ~|~ u >0 \}$. Note that we can show that this domain contains a uniform ball, by the $W^{2,p}_{loc}$ estimate on $u$. We can also show that its complement also contains an open ball. However, we do not have any a priori regularity of the boundary of $M^+$. Therefore, we instead consider the domain $M^+_\epsilon$ so that $M^+ \subset M^+_\epsilon$ and the boundary of $M^+_\epsilon$ is smooth. Heuristically, one should picture $M^+_\epsilon$ as being only slightly enlarged from $M^+$, but we will not need to use this explicitly. If we instead work with the Dirichlet problem on a smooth bounded domain, we can set $\Omega = M^+ = M^+_\epsilon$, and this step is unnecessary. 

At this point, we have reduced the original problem to studying the Dirichlet problem on a smooth open set in $M^+_\epsilon \subset M$. We now consider the drift $v^\prime$ which minimizes the principle eigenvalue $\lambda(\Delta u + v, M^+_\epsilon)$ among all drifts $v$ with $\| v \|_\infty < C$. Since $M^+_\epsilon$ is at least as large as $M^+$, $\lambda(\Delta u + v^\prime, M^+_\epsilon)$ is no greater than $\lambda(\Delta + v(\nabla \cdot), M^+)$. Therefore, it suffices for us to estimate $\lambda(\Delta u + v^\prime, M^+_\epsilon)$.

 We now consider the minimal principle eigenvalue $\lambda = \lambda(\Delta u + v^\prime, M^+_\epsilon)$ and its associated eigenfunction $u$, and prove that they satisfy the Dirichlet problem for the following semi-linear equation on $M^+_\epsilon$:
 \begin{equation} \label{Optimal drift}
 \Delta u + C | \nabla u| + \lambda u = 0
  \end{equation} 
 
 To do this, we assume that $v^\prime \neq C \frac{\nabla u}{|\nabla u|}$ on some subset of $M^+_\epsilon$ with non-zero measure. This implies that $u$ is a sub-solution to the following problem: 
 \begin{equation}
\Delta u - C\frac{\nabla u}{|\nabla u|} \cdot \nabla u + \lambda(\Delta u + v^\prime, M^+_\epsilon)   \leq \Delta u - v^\prime( \nabla u) \lambda u + \lambda(\Delta u + v^\prime, M^+_\epsilon) = 0      
 \end{equation} 

 Now, since $v^\prime$ is $L^\infty$ and $M^+_\epsilon$ is smooth, we have a local $W^{2,p}$ estimate on $u$, and hence $\nabla u$ is well defined. As such, $C\frac{\nabla u}{|\nabla u|}$ is $L^\infty$ and there exists a $W^{2,p}_{loc}$ solution to the Dirichlet problem.
 \begin{equation} \label{semi-linear equation} 0 = u^\prime - C\frac{\nabla u}{|\nabla u|} \cdot \nabla u^\prime + \lambda^\prime u^\prime \end{equation}
 
Since we assumed that $v^\prime$ minimizes $\lambda$, we know that $\lambda \leq  \lambda^\prime$ which implies that $u^\prime$ is a super-solution to the following problem:
\[ 0 \leq u^\prime - C\frac{\nabla u}{|\nabla u|} \cdot \nabla u^\prime + \lambda u^\prime \]

 Since $M^+_\epsilon$ is smooth and the drift is $L^\infty$, the Hopf lemma holds and shows that $\nabla u \neq 0$ on the boundary. From this, if we consider $u-\kappa u^\prime$, and choose $\kappa$ so that it is the maximum such $\kappa$ for which $u-\kappa u^\prime \geq 0$. From this, we can use a standard touching argument and either the maximum principle or the Hopf lemma to show that $u \equiv \kappa u^\prime$. In fact, this is exactly Lemma 2.1 of Hamel et al. \cite{HNR}, applied to an open domain on a manifold. As such, we have proven the ansatz.

This observation should be somewhat surprising. It shows that in the worst case scenario, where all the drift is working to make the principle eigenvalue as small as possible, we end up with much stronger regularity than we initially assumed. This gives us very strong control of the drift away from the zero locus of $u$ and $\nabla u$. In essence, all the drift is working together and cannot be too irregular. This phenomena was first observed in \cite{HNR}, which considered the drift-Laplacian with Dirichlet boundary conditions on $C^{2,\alpha}$ open domains in $\R^n$ and proved a version of the Faber-Krahn inequality.

\subsection{The a posteriori $C^\alpha$ drift estimate and uniform radii estimates}
\label{A posteriori drift estimates}

We now use our a priori regularity to ensure that the function $u$ does not vanish too quickly because we do not have $C^1$ control of the drift on the zero locus of $u$. To do this, we can use our a priori $C^{1,\alpha}$ estimate.

\subsubsection{Lipschitz estimates}

Define $p \in M$ to be an argmax of $u$ (i.e. satisfy $u(p)=1$). We define the c-radius $r_c$ as $\inf_x (d(x,p)~|~ u(x)= c,~ u(p)=1)$. For shorthand, we denote $d_{1-c} := \frac{1-c}{C_4}$. By the $C^{1,\alpha}$ estimate on $u$, we have $r_{c} > d_{1-c}$.

Intuitively, $d_c$ is the smallest distance we can travel to find an oscillation of $c$. This estimate only depends on the geometry of the manifold. Therefore, we can use the constant $d_c$ throughout the estimate. To calculate $d_c$ explicitly, note that we would have had to calculate $C_4$ explicitly.

\subsubsection{Higher regularity away from the zero locus of $\nabla u$}

From the $C^{1, \alpha}$-estimate on $u$, there is trivially a $C^\alpha$ estimate on $|\nabla u|$. Thus, when $|\nabla u|$ is non-zero, we have that $u$ satisfies $ \Delta u + C \frac{\nabla u}{|\nabla u|} \cdot \nabla u - \lambda u = 0$. The coefficients are now $C^\alpha$, so we gain $C^{2, \alpha}$ control on $u$ away from where $|\nabla u| =0$ by Schauder theory. Therefore, $|\nabla u| \in C^{1, \alpha}$ in this neighborhood and hence using the Schauder interior estimates again, we have that $u \in C^{3, \alpha}$ in a possible smaller neighborhood.

Away from the zero locus of $u$ and $\nabla u$, this bound allows us to take three derivatives of $u$, which is necessary to use Bochner's formula.
However, this estimate cannot be done uniformly as $u$ approaches $1$, and so we cannot use these bounds in our estimate of $ \lambda$; doing so makes argument circular when we try to choose the parameter $\beta$.


\subsection{The Li-Yau Estimate}

Now that we have $C^{3,\alpha}$ regularity of $u$ and the regularity of the drift away from a singular locus, we can apply the Li-Yau estimate. This step requires a fairly involved calculation, but the goal is to apply the maximum principle to a suitably chosen function to obtain a gradient estimate on $u$.

Recall that we have a function $u \in W^{2,p}(M^+_\epsilon)$ which satisfies the following:
\begin{equation}\label{eq:ADE}
\begin{cases}
 \Delta u + C|\nabla u| + \lambda u = 0 \\
 u \big \vert_{\partial M^+_\epsilon} \equiv 0
\end{cases}
\end{equation}

Suppose further that we have rescaled $u$ so that $\sup u = 1$ and that $\textrm{argmax}(u) = p$. We define $\rho(x) = dist(p,x)$ and fix a parameter $\beta >1$ to be determined later.

 We now consider the function $F(x)$ defined by: 
 \begin{equation}\label{eq:F} 
F(x)= \dfrac{|\nabla   u|^2}{(\beta-  u)^2} (r_{1/2}^2- \rho^2) 
 \end{equation}

 We observe that there is a point $x \in B_{r_{3/4}}$ with $|\nabla  (u)|> \frac{1}{4d}$ where $d$ is the diameter of $M$. At such a point, \begin{equation}
 \dfrac{|\nabla  (u)|^2}{(\beta- (u))^2} (r_{1/2}^2- \rho^2) > \dfrac{1}{16d^2 (\beta- 3/4)^2} (c r_{3/4})      
 \end{equation}

We consider the point $x_0 \in M$ which maximizes $F(x)$. Our previous estimate shows the following:
\begin{equation} \label{Gradient lower bound}
     |\nabla  (u)|^2> \frac{{(\beta- 1)^2} }{d^2} \dfrac{1}{16d^2 (\beta- 3/4)^2} (c r_{3/4})
\end{equation}
 
 

 Using the a priori $C^{1, \alpha}$ estimate on $u$, Inequality \ref{Gradient lower bound} shows that for fixed $\beta>1$, we can find a small ball $B$ around $x_0$ so that $|\nabla u| \neq 0$ in $B$. As described in Subsection \ref{A posteriori drift estimates}, Schauder theory allows us to bootstrap the regularity of $u$ to $C^{3,\alpha}$ in a small neighborhood around $x_0$. The size of this neighborhood will decay as $\beta$ gets close to 1. However, for a fixed $\beta$, this is enough regularity to apply the maximum principle. 
 
 

It is worthwhile to make some further remarks about this step. The function $F(x)$ incorporates the Lipschitz estimate of $u$ in its cut-off function. However, we do not directly use the a priori $C^\alpha$ continuity of $\nabla u$. That additional regularity is needed here, to ensure that $\nabla u$ does not vanish in a neighborhood of $x_0$ (so that the a posteriori Schauder estimates hold).

\subsubsection{Bochner's formula}

We consider an orthonormal frame around $x_0$. Recall that by our bound on the Ricci curvature, we have that $Ric(M) > -(n-1)K$ for some $K$.
\begin{eqnarray*} 
\Delta (| \nabla u|^2) & =& 2 \sum_{i,j} u_{ij}^2 + 2 \sum_i u_i (\Delta u)_i + 2 Ric( \nabla u, \nabla u) \\
			& =& 2 \sum_{i,j} u_{ij}^2 + 2 \sum_i u_i (-C|\nabla u| - \lambda u)_i + 2 Ric( \nabla u, \nabla u) \\
			&=& 2 \sum_{i,j} u_{ij}^2 + 2 \sum_i u_i (-C|\nabla u|  - \lambda u)_i + 2 Ric( \nabla u, \nabla u) \\
			& \geq& 2 \sum_{i,j} u_{ij}^2 + 2 \sum_i u_i (-C|\nabla u| - \lambda u)_i -(n-1)K| \nabla u|^2 \\
			& =& 2 \sum_{i,j} u_{ij}^2 + 2 \sum_i u_i (-C|\nabla u|)_i -((n-1)K+\lambda)| \nabla u|^2
\end{eqnarray*}

We may choose normal coordinates at $x$ so that $u_1(x_0) = |\nabla u|$, $u_i = 0$ for $i \neq 1$. This choice ensures that $ \nabla_j | \nabla u| = u_{1j}$ and hence $| \nabla ( | \nabla u|)|^2 = \sum_j u^2_{1j}$. We also have the following identity: 
 \[ \Delta (| \nabla u|^2) = 2 |\nabla u| \Delta (| \nabla u|) + 2 | \nabla ( | \nabla u|)|^2. \]

Substituting this equation into the preceding inequality, we find the following.
\begin{eqnarray*} 
 |\nabla u| \Delta (| \nabla u|) & \geq &  \sum_{i,j} u_{ij}^2 - \sum_j u^2_{1j} - 2 \sum_i u_i (C|\nabla u|)_i -((n-1)K+\lambda)| \nabla u|^2  
\end{eqnarray*}

We now estimate the first two terms. 
\begin{eqnarray*} 
 \sum_{i,j} u_{ij}^2 - \sum_j u^2_{1j} 
 & \geq &  \sum_{i>1} u^2_{i1} + \frac{1}{n-1} (\sum_{i>1} u_{ii})^2\\
   & \geq &  \sum_{i>1} u^2_{i1} + \frac{1}{n-1} (-C|\nabla u| - \lambda u - u_{11})^2\\
   & \geq &  \sum_{i>1} u^2_{i1} + \frac{1}{n-1} \left( \frac{u_{11}^2}{2} - 2(C|\nabla u|)^2 - 2(\lambda u)^2 \right)\\
  & \geq & \frac{1}{2(n-1)} | \nabla |\nabla u||^2 - \frac{2}{n-1} \left((C|\nabla u|)^2 + (\lambda u)^2 \right)
\end{eqnarray*}

This implies the following.
\begin{eqnarray*}
 \Delta (| \nabla u|^2) & \geq & \left( 2+ \frac{1}{(n-1)} \right) | \nabla |\nabla u||^2 - 2 \sum_i u_i (C|\nabla u|)_i  \\
 				 &  & -((n-1)K+\lambda)| \nabla u|^2 - \frac{2}{n-1} \left((C|\nabla u|)^2 + (\lambda u)^2 \right)
\end{eqnarray*}

\subsection{An estimate using the maximum principle}

We are now ready to estimate $F(x)$ using the Li-Yau estimate. Recall that we defined $F(x)$ in the following way.
 \[ F(x) = \frac{ | \nabla u|^2}{(\beta-u)^2}(r_{1/2}^2 - \rho^2) \]

Since $F_{\partial B_{r_{1/2}} (p)} \equiv 0$, we can find $x_0$ inside this ball where $F$ is maximized. We can assume that $x_0$ is not a cut point or else we can slightly alter our cut-off function as is done in \cite{SY}. Therefore, we assume that the cut-off function is smooth at this point.
\vspace{.3in}

 At $x_0$, we pick an orthonormal frame so that $u_1 = | \nabla u |$ and $u_i = 0$ for $i \neq 1$. Then, since $x_0$ maximizes $F$ (and $F$ is twice differentiable at $x_0$ by the Schauder estimate), we have that $\nabla F(x_0) = 0$, 
\begin{eqnarray*}
(r_{1/2}^2 - \rho^2) \frac{2u_{1i} u_1}{(\beta -u)^2} - 2 (r_{1/2}^2 - \rho^2) \frac{u_1^2 u_i}{(\beta -u)^3} - 2 \rho \rho_i \frac{ | \nabla u|^2}{(\beta-u)^2} = 0
\end{eqnarray*}

We can simplify this identity to obtain the following identities.
\begin{eqnarray}
 \label{u identities at x_0}
u_{i1} = u_{1i}  = \frac{ \rho \rho_i}{(r_{1/2}^2 - \rho^2)} \textrm{ for $i \neq 1$} \\
u_{11} = \frac{u_1^2}{(\beta -u)}+ \rho \rho_1 \frac{ | \nabla u|}{(r_{1/2}^2 - \rho^2)} \geq \frac{u_1^2}{(\beta -u)} -  \frac{ \rho | \nabla u|}{(r_{1/2}^2 - \rho^2)}\label{u11 identity at x_0}
\end{eqnarray}

We also have the following formula for the Laplacian of $F$.
\begin{eqnarray*}
 (\Delta F)  \frac{(\beta-u)^2}{(r_{1/2}^2 - \rho^2)} + (\nabla F) \nabla \left( \frac{(\beta-u)^2}{(r_{1/2}^2 - \rho^2)} \right) + F \Delta \left(\frac{(\beta-u)^2}{(r_{1/2}^2 - \rho^2)} \right) = \Delta (| \nabla u|^2)
\end{eqnarray*}

We now use this equation to prove the following estimate on $F$.

\begin{lemma}
Let $d$ is the diameter of $M^n$, $K$ a lower bound on the Ricci curvature and $C$ is the bound on the drift.
At the point $x_0$, we have the following estimate on $F$. 
   \begin{eqnarray*}
       0   & \geq & \left( \frac{1}{2(n-1)}  -  \frac{1}{4(n-1)^2} \right) F^2 - (4(n-1)-1) \left( 2+ \frac{1}{(n-1)} \right) 2 \rho F  \\
       & &  - 2 C d^8 F^{3/2} - 2 C d^7 F  -((n-1)K+\lambda)F d^4 \\
  &  & - \frac{2}{n-1} (\lambda u)^2 \frac{ d^8}{(\beta -u)^2} -\frac{2}{n-1} C^2 Fd^6\\
  & &  -F(  \lambda u) \frac{d^6}{(\beta -u)} - F^{3/2} C d^5\\
   &  & -8 F^{3/2} d^4 - 2(n-1)(1+ K \rho)F d^4
\end{eqnarray*}

\end{lemma}
\begin{proof}

The proof of this lemma is a very long string of manipulations combined with the use of the Laplace comparison theorem. We start by noting that at $x_0$, $\Delta F \leq 0$ and $\nabla F = 0$, which allows us to use our previous identities and inequalities.


\begin{eqnarray*}
0 & \geq & \Delta (| \nabla u|^2) - F \Delta \left(\frac{(\beta-u)^2}{(r_{1/2}^2 - \rho^2)} \right) \\
  &\geq &\left( 2+ \frac{1}{(n-1)} \right) | \nabla |\nabla u||^2 - 2 \sum_i u_i (C|\nabla u| )_i -((n-1)K+\lambda)| \nabla u|^2 \\
  &  & - \frac{2}{n-1} \left((C|\nabla u| )^2 + (\lambda u)^2 \right) - F \Delta \left(\frac{(\beta-u)^2}{(r_{1/2}^2 - \rho^2)} \right) \\
   &\geq &\left( 2+ \frac{1}{(n-1)} \right) \left( |u_{11}|^2+ \sum_{ i \neq 1}|u_{i1}|^2 \right)  - 2 u_1 (C|\nabla u| )_1  \\
  &  & -((n-1)K+\lambda)| \nabla u|^2 - \frac{2}{n-1} \left((C|\nabla u| )^2 + (\lambda u)^2 \right) - 2 F \frac{|\nabla u|^2}{(r_{1/2}^2 - \rho^2)}  \\
   & &- 2 F (\beta -u) \frac{\Delta u}{(r_{1/2}^2 - \rho^2)} +8 F (\beta -f) \rho \frac{\nabla u \cdot \nabla \rho}{(r_{1/2}^2 - \rho^2)^2} - F (\beta -u)^2 \Delta (r_{1/2}^2 - \rho^2)^{-1} 
     \end{eqnarray*}
   \begin{eqnarray*}
   &\geq &\left( 2+ \frac{1}{(n-1)} \right) \left[ \left( \frac{|u_{1}|^2}{\beta -u} + 2 \frac{\rho_1 \rho u_1}{(r_{1/2}^2 - \rho^2)} \right)^2+ \sum_{ i \neq 1}|u_{i1}|^2 \right] \\
    & & - 2 u_1 (C|\nabla u| )_1 -((n-1)K+\lambda)| \nabla u|^2 \\
  &  & - \frac{2}{n-1} \left((C|\nabla u| )^2 + (\lambda u)^2 \right) - 2 F \frac{|\nabla u|^2}{(r_{1/2}^2 - \rho^2)} - 2 F (\beta -u) \frac{\Delta u}{(r_{1/2}^2 - \rho^2)} \\
   & &+8 F (\beta -u) \rho \frac{\nabla u \cdot \nabla \rho}{(r_{1/2}^2 - \rho^2)^2} - F (\beta -u)^2 \Delta (r_{1/2}^2 - \rho^2)^{-1}  
   \end{eqnarray*}
   \begin{eqnarray*}
      &\geq &\left( 2+ \frac{1}{(n-1)} \right) \left( 1 -  \frac{1}{4(n-1)} \right) \frac{|u_{1}|^4}{(\beta -u)^2} \\
       & & - (4(n-1)-1) \left( 2+ \frac{1}{(n-1)} \right) \left( 2 \frac{\rho \rho_1  u_1}{(r_{1/2}^2 - \rho^2)} \right)^2 \\
       & & + \left( 2+ \frac{1}{(n-1)} \right) \sum_{ i \neq 1}|u_{i1}|^2  - 2  u_1 (C|\nabla u| )_1 -((n-1)K+\lambda)| \nabla u|^2 \\
  &  & - \frac{2}{n-1} \left((C|\nabla u| )^2 + (\lambda u)^2 \right) - 2 F \frac{|\nabla u|^2}{(r_{1/2}^2 - \rho^2)} - 2 F (\beta -u) \frac{\Delta u}{(r_{1/2}^2 - \rho^2)} \\
   & &+8 F (\beta -u) \rho \frac{\nabla u \cdot \nabla \rho}{(r_{1/2}^2 - \rho^2)^2} - F (\beta -u)^2 \Delta (r_{1/2}^2 - \rho^2)^{-1}  
   \end{eqnarray*}
   
  Recalling that $2 F \frac{|\nabla u|^2}{(r_{1/2}^2 - \rho^2)} = 2 \frac{|u_{1}|^4}{(\beta -u)^2}$, we see that this term partially cancels out the firs term in the previous inequality. Making this substitution and simplifying other terms, we find the following:
\begin{eqnarray*}
0 & \geq & \left( \frac{1}{2(n-1)}  -  \frac{1}{4(n-1)^2} \right) \frac{|u_{1}|^4}{(\beta -u)^2}  \\
& &- (4(n-1)-1) \left( 2+ \frac{1}{(n-1)} \right) \left( 2 \frac{\rho \rho_1 u_1}{(r_{1/2}^2 - \rho^2)}\right)^2 \\
       & & + \left( 2+ \frac{1}{(n-1)} \right) \sum_{ i \neq 1}\frac{ \rho^2 \rho^2_i |u_1|^2}{(r_{1/2}^2 - \rho^2)^2}  - 2  u_1 (C|\nabla u| )_1  -((n-1)K+\lambda)| \nabla u|^2 \\
  &  & - \frac{2}{n-1} \left((C|\nabla u| )^2 + (\lambda u)^2 \right)  - 2 F (\beta -u) \frac{\Delta u}{(r_{1/2}^2 - \rho^2)}  \\
   & &+8 F (\beta -u) \rho \frac{\nabla u \cdot \nabla \rho}{(r_{1/2}^2 - \rho^2)^2}   - F (\beta -u)^2 \Delta (r_{1/2}^2 - \rho^2)^{-1}
\end{eqnarray*}

Substituting in $\Delta u + C|\nabla u| + \lambda u = 0$ into the fourth line and then simplifying, this yields:
\begin{eqnarray*}
0 
   & \geq & \left( \frac{1}{2(n-1)}  -  \frac{1}{4(n-1)^2} \right) \frac{|u_{1}|^4}{(\beta -u)^2} \\ 
   & & - (4(n-1)-1) \left( 2+ \frac{1}{(n-1)} \right)( 2 \frac{\rho u_1}{(r_{1/2}^2 - \rho^2)})^2 \\
       & &  - 2 u_1 (C|\nabla u| )_1 -((n-1)K+\lambda)| \nabla u|^2 \\
  &  & - \frac{2}{n-1} \left((C|\nabla u| )^2 + (\lambda u)^2 \right) - 2 F (\beta -u) \frac{-C|\nabla u|  - \lambda u}{(r_{1/2}^2 - \rho^2)} \\
   & &+8 F (\beta -u) \rho \frac{ u_1 \rho_1}{(r_{1/2}^2 - \rho^2)^2} - F (\beta -u)^2 \Delta (r_{1/2}^2 - \rho^2)^{-1}  
\end{eqnarray*}

Multiplying both sides of the inequality by $\frac{ (r_{1/2}^2 - \rho^2)^4}{(\beta -u)^2}$, and substituting in the definition of $F$ in the last line, we have the following.
\begin{eqnarray*}
0 
      & \geq & \left( \frac{1}{2(n-1)}  -  \frac{1}{4(n-1)^2} \right) F^2 - (4(n-1)-1) \left( 2+ \frac{1}{(n-1)} \right) 2 \rho F  \\
       & &  - 2 \frac{ (r_{1/2}^2 - \rho^2)^4}{(\beta -u)^2} u_1 (C|\nabla u| )_1 -((n-1)K+\lambda)F (r_{1/2}^2 - \rho^2)^2 \\
  &  & - \frac{2}{n-1} \left((C|\nabla u| )^2 + (\lambda u)^2 \right) \frac{ (r_{1/2}^2 - \rho^2)^4}{(\beta -u)^2} +F( C |\nabla u|  + \lambda u) \frac{ (r_{1/2}^2 - \rho^2)^3}{(\beta -u)}\\
   &  & -8 F^{3/2}  \rho(r_{1/2}^2 - \rho^2)^{3/2} - F \Delta (r_{1/2}^2 - \rho^2)^{-1}(r_{1/2}^2 - \rho^2)^4
   \end{eqnarray*}
   
   Using Identity \eqref{u11 identity at x_0} and further simplifying, we find that
   \begin{eqnarray*}
       0   
   & \geq & \left( \frac{1}{2(n-1)}  -  \frac{1}{4(n-1)^2} \right) F^2 - (4(n-1)-1) \left( 2+ \frac{1}{(n-1)} \right) 2 \rho F  \\
       & &  - 2 \frac{ (r_{1/2}^2 - \rho^2)^4}{(\beta -u)^3} u_1^3 C - 2 \frac{ \rho (r_{1/2}^2 - \rho^2)^3}{(\beta -u)^2} u_1^2 C \\
       & & -((n-1)K+\lambda)F (r_{1/2}^2 - \rho^2)^2 \\
  &  & - \frac{2}{n-1} (\lambda u)^2 \frac{ (r_{1/2}^2 - \rho^2)^4}{(\beta -u)^2} -\frac{2}{n-1} C^2 F(r_{1/2}^2 - \rho^2)^3\\
  & &  + \lambda u F  \frac{ (r_{1/2}^2 - \rho^2)^3}{(\beta -u)} + F^{3/2} C  (r_{1/2}^2 - \rho^2)^{5/2}\\
   &  & -8 F^{3/2}  \rho(r_{1/2}^2 - \rho^2)^{3/2} - F \Delta (r_{1/2}^2 - \rho^2)^{-1}(r_{1/2}^2 - \rho^2)^4
\end{eqnarray*}

Therefore, we have the following inequality.
   \begin{eqnarray}  \label{Laplacecomparisonterm}
    0  & \geq & \left( \frac{1}{2(n-1)}  -  \frac{1}{4(n-1)^2} \right) F^2 - (4(n-1)-1) \left( 2+ \frac{1}{(n-1)} \right) 2 \rho F  \\
       & &  - 2 C (r_{1/2}^2 - \rho^2)^4 F^{3/2} - 2 C \rho (r_{1/2}^2 - \rho^2)^3 F  -((n-1)K+\lambda)F (r_{1/2}^2 - \rho^2)^2 \nonumber \\
  &  & - \frac{2}{n-1} (\lambda u)^2 \frac{ (r_{1/2}^2 - \rho^2)^4}{(\beta -u)^2} -\frac{2}{n-1} C^2 F(r_{1/2}^2 - \rho^2)^3 \nonumber \\
  & &  +F(  \lambda u) \frac{ (r_{1/2}^2 - \rho^2)^3}{(\beta -u)} + F^{3/2} C  (r_{1/2}^2 - \rho^2)^{5/2} \nonumber\\
   &  & -8 F^{3/2}  \rho(r_{1/2}^2 - \rho^2)^{3/2} - F \Delta (r_{1/2}^2 - \rho^2)^{-1}(r_{1/2}^2 - \rho^2)^4 \nonumber 
\end{eqnarray}

We now focus our efforts into estimating the final term of Inequality \eqref{Laplacecomparisonterm}. Using the Laplace comparison theorem (as in \cite{SY}), we have the following inequality.
   \begin{eqnarray*} 
    \Delta (r_{1/2}^2 - \rho^2)^{-1} &=& \sum_i  \frac{2 \rho_{ii} \rho}{(r_{1/2}^2 - \rho^2)^2} + \frac{8 \rho^2_{i} \rho^2}{(r_{1/2}^2 - \rho^2)^3}+ \frac{2 \rho^2_{i}}{(r_{1/2}^2 - \rho^2)^2} \\
    & \leq & \frac{n-1}{\rho}(1+ K \rho) \frac{2 \rho}{(r_{1/2}^2 - \rho^2)^2} + \frac{8 \rho^2}{(r_{1/2}^2 - \rho^2)^3}+ \frac{2}{(r_{1/2}^2 - \rho^2)^2} \nonumber
   \end{eqnarray*}

This yields the following estimate on the last term in Inequality \eqref{Laplacecomparisonterm}.
  \begin{eqnarray} \label{Laplace comparison estimate}
  F \Delta (r_{1/2}^2 - \rho^2)^{-1}(r_{1/2}^2 - \rho^2)^4 & \leq &  2(n-1)(1+ K \rho)F (r_{1/2}^2 - \rho^2)^2 + 8F \rho^2(r_{1/2}^2 - \rho^2) \\
  & &+ 2F(r_{1/2}^2 - \rho^2)^2  \nonumber
     \end{eqnarray}

Combining Inequalities \eqref{Laplacecomparisonterm} and \eqref{Laplace comparison estimate}, we find the following.
      \begin{eqnarray} \label{before using diameter bound}
       0      & \geq & \left( \frac{1}{2(n-1)}  -  \frac{1}{4(n-1)^2} \right) F^2 - (4(n-1)-1) \left( 2+ \frac{1}{(n-1)} \right) 2 \rho F  \\
       & &  - 2 C (r_{1/2}^2 - \rho^2)^4 F^{3/2} - 2 C \rho (r_{1/2}^2 - \rho^2)^3 F  -((n-1)K+\lambda)F (r_{1/2}^2 - \rho^2)^2 \nonumber \\
  &  & - \frac{2}{n-1} (\lambda u)^2 \frac{ (r_{1/2}^2 - \rho^2)^4}{(\beta -u)^2} -\frac{2}{n-1} C^2 F(r_{1/2}^2 - \rho^2)^3 \nonumber \\
  & &  +F(  \lambda u) \frac{ (r_{1/2}^2 - \rho^2)^3}{(\beta -u)} + F^{3/2} C  (r_{1/2}^2 - \rho^2)^{5/2} \nonumber \\
   &  & -8 F^{3/2}  \rho(r_{1/2}^2 - \rho^2)^{3/2} \nonumber \\
   & & - 2(n-1)(1+ K \rho)F (r_{1/2}^2 - \rho^2)^2 + 8F \rho^2(r_{1/2}^2 - \rho^2)+ 2F(r_{1/2}^2 - \rho^2)^2 \nonumber
\end{eqnarray}

Denoting the diameter of $M$ by $d$, we have that $r_{1/2}$ and  $\rho < d $. As a result, $(r_{1/2}^2 - \rho^2) < d^2$. Substituting these into Inequality \eqref{before using diameter bound} and dropping the final two positive terms, we find the desired inequality.
  \begin{eqnarray*}
       0      & \geq & \left( \frac{1}{2(n-1)}  -  \frac{1}{4(n-1)^2} \right) F^2 - (4(n-1)-1) \left( 2+ \frac{1}{(n-1)} \right) 2 \rho F  \\
       & &  - 2 C d^8 F^{3/2} - 2 C d^7 F  -((n-1)K+\lambda)F d^4 \\
  &  & - \frac{2}{n-1} (\lambda u)^2 \frac{ d^8}{(\beta -u)^2} -\frac{2}{n-1} C^2 Fd^6\\
  & &  -F(  \lambda u) \frac{d^6}{(\beta -u)} - F^{3/2} C d^5\\
   &  & -8 F^{3/2} d^4 - 2(n-1)(1+ K \rho)F d^4
\end{eqnarray*}
\end{proof}


\subsection{Using the inequality on $F$}

At this point, we take stock of the situation to show how this gives any hope of providing a lower bound on $\lambda$.
We have uniform control of the cutoff function in $B_{r_{3/4}}$ from the a priori gradient estimate on $u$. This allows us to change the inequality on $F$ to obtain to an inequality on $\dfrac{|\nabla   u|}{(\beta-  u)}$. If we integrate out this inequality along a geodesic from $p$ to some $x$ satisfying $ u(x)=\frac34$, we obtain a bound of the form 
\[\log\left( \frac{ \beta - 3/4}{\beta-1} \right) \leq \left( \mathcal{C} + c\frac{\lambda^{1/2}}{(\beta - 1)^{1/2}} \right), \]
where $\mathcal{C}$ and $c$ and constants to be determined.

For $\beta$ close to 1, the left hand side blows up, which shows that right hand side must blow up as well and implies a lower bound on $\lambda$.
It requires some care to to make this precise. In particular, as $\beta$ goes to 1, our $C^{3, \alpha}$ control on $u$ weakens. Therefore, we set $\beta - 1$ at some small but fixed scale. In order for this to work, it is crucial that the constants $\mathcal{C}$ and $c$ only depend on the a priori bounds, so are independent of $\beta$. This allows us to maintain enough control to use the Bochner identity and the maximum principle while still being free to pick $\beta$ in a way that yields a positive bound on $\lambda$. 

We now do this precisely. For convenience, we denote $f := \sqrt{ F}$ and rewrite the previous inequality in terms of $f$.
      \begin{eqnarray*}
       0      & \geq & \left( \frac{1}{2(n-1)}  -  \frac{1}{4(n-1)^2} \right) f^4 - (4(n-1)-1) \left( 2+ \frac{1}{(n-1)} \right) 2 d f^2  \\
       & &  - 2 C d^8 f^3 - 2 C d^7 f^2  -((n-1)K+\lambda) d^4 f^2 \\
  &  & - \frac{2}{n-1} (\lambda u)^2 \frac{ d^8}{(\beta -u)^2} -\frac{2}{n-1} C^2 f^2 d^6\\
  & &  -f^2(  \lambda u) \frac{d^6}{(\beta -u)} - f^3 C d^5\\
   &  & -8 f^3 d^4 - 2(n-1)(1+ K \rho)f^2 d^4 
\end{eqnarray*}

For conciseness, we denote $\alpha = \frac{1}{\beta-1}$ and observe that $\alpha > \frac{u}{\beta-u}$. We also define the following constants:
 $$A= \left( \frac{1}{2(n-1)}  -  \frac{1}{4(n-1)^2} \right),$$ $$B =2 C d^8+ C d^5+8 d^4,$$
 $$\mathcal{ D} =   \lambda \frac{d^6}{(\beta -1)},$$ 
\begin{eqnarray*}
D &= &  (4(n-1)-1) \left( 2+ \frac{1}{(n-1)} \right) 2 d + 2 C d^7+ ((n-1)K+\lambda) d^4 \\
 & &+ \frac{2}{n-1} C^2 d^6 + 2(n-1)(1+ K \rho) d^4
\end{eqnarray*}
 $$\mathcal{E} =  \frac{2\lambda^2}{n-1}  \frac{ d^8}{(\beta -1)^2}$$

From the previous inequality, we have the following estimate:
      \begin{eqnarray*}
       0      & \geq & A f^4 - B f^3 - D f^2 - \mathcal{D} f^2 - \mathcal{E} 
\end{eqnarray*}

Note that the calligraphic terms are the only terms where the coefficients aren't uniform in $\beta$ and these both contain a $\lambda$.
Now we use a lemma about the roots of quartics. This lemma was originally proven in \cite{GK}.

  \begin{lemma}\label{Roots of quartics}
Suppose $A_1,A_2, A_3  >0$ and x satisfies $P(x) = x^4 - A_1x^3 -A_2x^2 - A_3 \leq 0$.
 Then $x \leq A_1 + \sqrt{A_2 +\sqrt{A_3}}$.
\end{lemma}

 In order to make future calculations more feasible, we note that the following inequality holds:
$$A_1 + \sqrt{A_2 +\sqrt{A_3}}  \leq A_1 +  (2A_2)^{1/2} +(4A_3)^{1/4}$$

Applying this inequality to $f$, this shows that
\begin{eqnarray*} 
f &  \leq & \frac{1}{A} \left(B + \sqrt{2} \sqrt{ D + \mathcal{D}} + \sqrt{2} \mathcal{E}^{1/4} \right) \\
 &  \leq & \frac{1}{A} \left(B + 2 \sqrt{ D} \right) +  \frac{1}{A} \left( 2\sqrt{\mathcal{D}} +  \sqrt{2} \mathcal{E}^{1/4}   \right)
\end{eqnarray*} 

Using the fact that $n \geq 2$, we obtain the following simplified estimates.
\begin{eqnarray*} 
f  &  \leq & \frac{1}{A} \left(B + 2 \sqrt{ D} \right) +  8(n-1) \left( d^3 +  d^2  \right) \frac{\sqrt{\lambda}}{\sqrt{\beta -1}}
\end{eqnarray*} 

From the $C^{1,\alpha}$ estimate, we have $ r_{1/2} \geq r_{3/4} + d_{1/4}$ and so in $B_{r_{3/4}},$ 
\begin{eqnarray*} 
(r_{1/2}^2 - \rho^2)  \geq   (r_{1/2}^2 - r_{3/4}^2) \geq 3 d^2_{1/4}.
\end{eqnarray*}

Using the definition of $f$, this implies that in $B_{r_{3/4}}$, the following estimate holds:
\begin{eqnarray*} 
\frac{|\nabla u|}{\beta - u} \geq   \frac{1}{3 A  d^2_{1/4}} \left(B + 2 \sqrt{ D} \right) +  \frac{8(n-1)}{3d^2_{1/4}} \left( d^3 +  d^2  \right) \frac{\sqrt{\lambda}}{\sqrt{\beta -1}}
\end{eqnarray*}

Setting $\mathcal{C}= \frac{1}{3 A  d^2_{1/4}} \left(B + 2 \sqrt{ D} \right) $ and $c=  \frac{8(n-1)}{3d^2_{1/4}} \left( d^3 +  d^2  \right)$, this shows that
\begin{eqnarray*} 
\frac{|\nabla u|}{\beta - u} \leq   \mathcal{C} +  c \frac{\sqrt{\lambda}}{\sqrt{\beta -1}}.
\end{eqnarray*}

\subsection{A lower bound on $\lambda$}

We pick $x \in B_{r_{3/4}}$ with $u(x) = \frac{3}{4}$ and a minimal geodesic $\gamma$ between $x$ and $p$ (recall that $p$ is the point where $u(p)=1$). This integral is well defined because $u$ has enough continuity for $\nabla u$ to be defined pointwise.

We can estimate this integral in the following way.
\begin{eqnarray*} 
\log \frac{\beta - 3/4}{\beta - 1} \leq \int_{\gamma} \frac{|\nabla u|}{\beta - u} \leq  \left( \mathcal{C} +  c \frac{\sqrt{\lambda}}{\sqrt{\beta -1}} \right)   d
\end{eqnarray*}

Solving for $\lambda$, we have the following inequality.
\begin{equation} \label{semifinal inequality}
    \sqrt{\lambda} \geq  \frac{\sqrt{(\beta - 1)}}{c}\left( \frac{1}{d}  \log \frac{\beta - 3/4}{\beta - 1} -\mathcal{C} \right)
\end{equation}

To find a lower bound for $\lambda$, we must find a particular value for $\beta$ so that the right hand side of Inequality \eqref{semifinal inequality} is positive. To do so, we let $x =   \frac{\beta - 3/4}{\beta - 1}$ and set $x = e^{dC+d}$. This then shows the following, which finishes the proof the theorem.
\begin{equation} \label{final inequality}
    \lambda \geq \frac{1}{4c^2} (e^{d\mathcal{C}+d}-1)^{-1}
\end{equation}

\section{A concrete example}

The preceding argument establishes a lower bound on the real eigenvalues, but it does not make clear how the drift affects the spectrum. To illustrate how the size of the drift affects small eigenvalues, we calculate the minimal eigenvalue explicitly in a simple case.
Consider the circle $\mathbb{R}/4\mathbb{Z}$ and the problem 
\[ u^{\prime \prime} + f \cdot u^{ \prime} + \lambda u =0 \textrm{ with } \| f \|_{L^{\infty}} \leq C \]

To simplify the calculation, we set $C= 2b$. The goal is to minimize $\lambda$ under the constraint that $\| f \|_{\infty} \leq 2b$. By symmetry, we can instead consider the principle eigenvalue of the Dirichlet problem on the domain $[-1,1]$. Using symmetry and the drift ansatz, $u$ will satisfying the following ordinary differential equation: 
\begin{equation} \label{ODEeigenvalue}
    u^{\prime \prime} + 2b \cdot u^{ \prime} + \lambda u =0
\end{equation}
on the domain $[0,1]$ with the constraints:
\begin{enumerate}
\item $u(0)=1$,
\item $u^\prime(0) = 0$,
\item $u(1) = 0$, and
\item $\lambda$ is the minimal value so that Equation \ref{ODEeigenvalue} has a solution.
\end{enumerate}


This is a linear ODE, so we can solve it explicitly and then solve for $\lambda$ in terms of $b$ so that the boundary conditions are satisfied. Doing so, we find the following. 
\begin{enumerate}
    \item   For $b$ large, the principle eigenvalue satisfies the following equation:
\[ \frac{1+ e^{-2 \sqrt{b^2 -\lambda}}}{1- e^{-2 \sqrt{b^2 -\lambda}}} \sqrt{1-\lambda/b^2} =1 \]
\item  For $b$ small, $\lambda $ satisfies the equation:
\[ \sqrt{ \lambda - b^2} = b \tan \left( \sqrt{ \lambda - b^2} \right) \]
\end{enumerate}

Note that the first expression decreases rapidly as $b$ gets large. More precisely, the minimal eigenvalue satisfies the asymptotic
\[\lambda \approx 4b^2 e^{-2b} = C^2 e^{-C}. \]

\begin{center}
\includegraphics[width=120mm,scale=1]{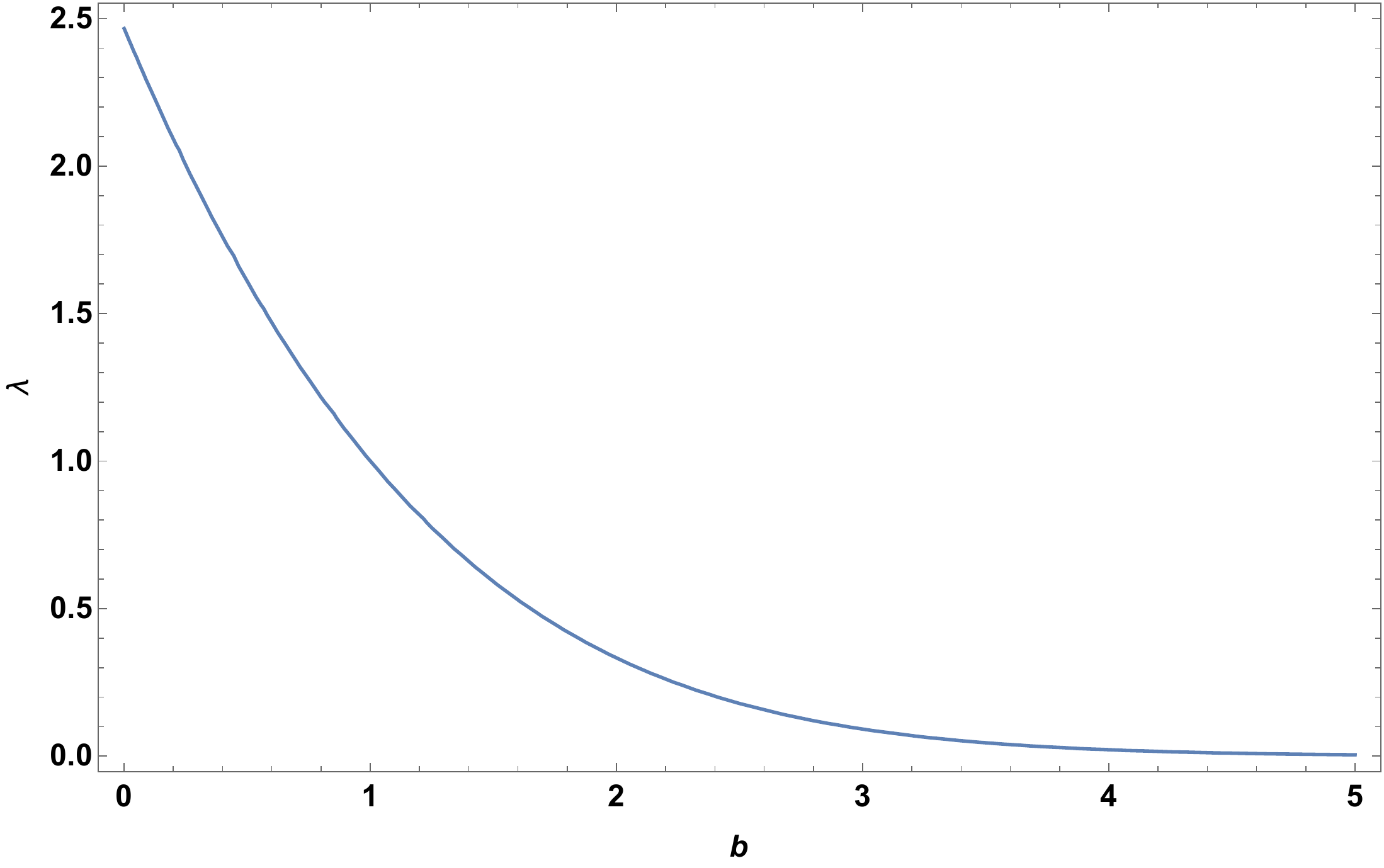}\\
A graph of $\lambda$ in terms of $b$
\end{center}


\section{Open Questions}
\label{Future work}
In this section, we discuss two further questions which arose while writing this manuscript.

\subsection{$L^p$ drifts}

One natural question is whether a similar result can be shown when the drift has $L^p$ bounds for $p < \infty$. For this question, it is necessary to assume that $p >n$. Firstly, we need this assumption in order to apply the Calder\'on-Zygmund estimates. More importantly, on the interval (i.e. $n=1$), it is possible to find $L^1$ drifts with arbitrarily small principle eigenvalue. Interestingly, when the drift is sufficiently small in $L^1$ norm, it is possible to recover an estimate on the eigenvalue by applying Gr\"onwall's inequality. As such, the minimal eigenvalue displays interesting threshold phenomena; it is positive for small drifts but as soon as the $L^1$ norm of the drift is sufficiently large, it can be arbitrarily small. A similar phenomena likely occurs when $p=n$ in higher dimensions as well, but there is no analog of the Gr\"onwall inequality to prove this.

For $n<p< \infty$, the main obstruction to repeating Theorem \ref{Main theorem} is the lack of a drift ansatz. It is possible to find a sequence of minimizing drifts, but the minimizer does not have any natural additional regularity. In particular, the corresponding eigenfunction is not a subsolution to a semi-linear equation independent of the choice of drift, which was the key idea that we used in the $L^\infty$ case.

One possible approach to this problem for domains in $\mathbb{R}^n$ is to convolve the eigenfunction with a suitable bump function. Using the $C^{1,\alpha}$ estimate and the increased regularity from convolution, it might be possible to prove that the convolved eigenfunction is a subsolution to an equation with more regular drift. If so, it might be possible to adapt our approach to obtain lower bounds in this setting.


\subsection{Relaxing the geometric control}

 In the statement of Theorem \ref{Main theorem}, there is a somewhat awkward two-sided bound on the Ricci curvature. Ideally, this could be replaced with a lower bound on the Ricci curvature alone. We expect this is the case as the only place such bounds are used is to bound the $C^{1,\alpha}$ harmonic radius from below. 
 
 Going further, it may also be possible to remove the dependence on the injectivity radius. The original Li-Yau estimate does not involve the injectivity radius and intuitively speaking, shrinking the injectivity radius would seem to increase, not decrease the eigenvalues. In order to prove an estimate along these lines with no assumptions on the injectivity radius, one would need to find a different way to prove the a priori regularity. As such, we put for the following conjecture.
 
 \begin{conjecture}
Let $(M^n, g)$ be a compact Riemannian manifold satisifying $Ric(M) > K$ and  $v$ be a one-form with $\| v \|_\infty < C$. Suppose that there exists $u \in W^{2,p} (M)$ satisfying $\Delta u + v(\nabla u) = \lambda u$ with $\lambda$ real. Then there exists some constant $\delta>0$ depending only on $K,~ C,~ diam(M),~inj(M)$ and $n$ so that $\lambda > \delta$.
\end{conjecture}

\bibliography{bibfile}
\bibliographystyle{alpha}


\end{document}